\documentclass[12pt]{article}%
\usepackage{amsmath, amsfonts, amsthm, color,latexsym}
\usepackage{amsmath}
\usepackage{amsfonts}
\usepackage{amssymb}
\usepackage{color} %color,soul
\allowdisplaybreaks[4]
\usepackage{graphicx}%
\providecommand{\U}[1]{\protect\rule{.1in}{.1in}}
\newtheorem{teo}{Theorem}[section]
\newtheorem{prop}[teo]{Proposition}
\newtheorem{cor}[teo]{Corollary}
\newtheorem{ex}[teo]{Example}
\newtheorem{obs}[teo]{Remark}

\newtheorem{lema}[teo]{Lemma}
\newtheorem{final remark}[teo]{Final Remark}
\newtheorem{definition}[teo]{Definition}
\newcommand{\an}{\left \Vert} % ABRE NORMA

\newcommand{\fn}{\right \Vert} % FECHA NORMA

\newcommand{\ap}{\left (} % ABRE PARENTESIS

\newcommand{\fp}{\right )} % FECHA PARENTESIS

\begin{document}

\title{\sc Constructing hyper-ideals of multilinear operators between Banach spaces}
\date{}
\author{Geraldo Botelho\thanks{Supported by CNPq Grant
305958/2014-3 and Fapemig Grant PPM-00326-13.\hfill\newline2010 Mathematics Subject
Classification: 47L20, 47B10, 46G25, 47L22. \newline Keywords: Banach spaces, multilinear operators, hyper-ideals, $\cal I$-bounded sets.}  ~and Ewerton R. Torres}\maketitle

\begin{abstract} In view of the fact that some classical methods to construct multi-ideals fail in constructing hyper-ideals, in this paper we develop two new methods to construct hyper-ideals of multilinear operators between Banach spaces. These methods generate new classes of multilinear operators and show that some important well studied classes are Banach or $p$-Banach hyper-ideals.
\end{abstract}

\section{Introduction and background}
Ideals of multilinear operators between Banach spaces (or simply multi-ideals), which happen to be classes of multilinear operators that are stable with respect to the composition with linear operators, were introduced by Pietsch \cite{pietsch} as a first attempt to extend the successful theory of ideals of linear operators (operator ideals) to the nonlinear setting. The theory of multi-ideals turned out be successful itself, and a refinement of this concept, called hyper-ideals, was introduced in \cite{ewerton} according to the following philosophy: the nonlinearity of the multilinear setting is better explored by considering classes of multilinear operators that are stable with respect to the composition with multilinear operators -- whenever this composition is possible, of course -- rather than with linear operators.

The basics of the theory of hyper-ideals and plenty of distinguished examples can be found in \cite{ewerton}. As the theory of hyper-ideals is quite more restrictive than the theory of multi-ideals, it is expected that some techniques do not pass from multi-ideals to hyper-ideals. This is exactly what happens with some general methods to construct multi-ideals. While the technique concerning composition ideals works nicely for hyper-ideals (see \cite[Theorem 4.2]{ewerton}), the factorization and the linearization methods are helpless in the realm of hyper-ideals (for a description of such methods, see, e.g. \cite{botelho1, B-Jii}). An illustration of the failure of these methods in the generation of hyper-ideals can be found in Example \ref{ex1}. The purpose of this paper is to fill this gap by developing two general methods to construct hyper-ideals.

In Section 2 we introduce a method based on the transformation of finite vector-valued sequences by multilinear operators. We show that this method, which is akin to different sorts of {\it summing multilinear operators}, gives rise to new classes and recovers, as a particular instance, the important class of strongly summing multilinear operators. In Section 3 we show that, proceeding for multilinear operators as Aron and Rueda \cite{aronrueda2} did for homogeneous polynomials, we end up with Banach hyper-ideals. In this fashion, classical multi-ideals, such as compact, weakly compact and $p$-compact multilinear operators, are shown to be Banach hyper-ideals.

Along the paper, $n$ is a positive integer, $E, E_n, F, G, G_n, H$, shall be Banach spaces over $\mathbb{K} = \mathbb{R}$ or $\mathbb{C}$, $B_E$ denotes the closed unit ball of $E$, $E'$ denotes the topological dual of $E$ and ${\mathcal{L}(E_1,\ldots,E_n;F)}$ stands for the Banach space of continuous $n$-linear operators from $E_1 \times \cdots \times E_n$ to $F$ endowed with the usual sup norm. When $F = \mathbb{K}$, the space of continuous $n$-linear forms is denoted by   ${\mathcal{L}(E_1,\ldots,E_n)}$. To avoid ambiguity, the space of $n$-linear forms on $\mathbb{K}$ shall be denoted by ${\cal L}(^n\mathbb{K};\mathbb{K})$. Given functionals $\varphi_1 \in E_1', \ldots, \varphi_n \in E_n'$ and a vector $b \in F$, $\varphi_1 \otimes \cdots \otimes \varphi_n \otimes b$ denotes the element of ${\mathcal{L}(E_1,\ldots,E_n;F)}$ given by
$$\varphi_1 \otimes \cdots \otimes \varphi_n \otimes b(x_1, \ldots, x_n) = \varphi_1(x_1) \cdots \varphi_n(x_n)b. $$
The elements of the subspace of ${\mathcal{L}(E_1,\ldots,E_n;F)}$ generated by operators of the form $\varphi_1 \otimes \cdots \otimes \varphi_n \otimes b$ are called multilinear operators of finite type. A linear space-valued map is of finite rank if its range generates a finite dimensional subspace of the target space.

Normed, $p$-normed and Banach ideals of linear operators (operator ideals) are always meant in the sense of \cite{klauslivro, pietschlivro}. According to \cite{ewerton}, a \textit{hyper-ideal of multilinear operators}, or simply a \textit{hyper-ideal}, is a subclass $\mathcal{H}$ of the class of all continuous multilinear operators between Banach spaces such that all components $$\mathcal{H}(E_1,\ldots, E_n;F):=\mathcal{L}(E_1,\ldots, E_n;F)\cap \mathcal{H}$$
satisfy:\\
$(1)$ $\mathcal{H}(E_1,\ldots, E_n;F)$ is a linear subspace of $\mathcal{L}(E_1,\ldots, E_n;F)$ which contains the $n$-linear operators of finite type;\\
$(2)$ {\bf The hyper-ideal property:} Given natural numbers $n$ and $1\le m_1<\cdots<m_n$, and  Banach spaces $G_1,\ldots,G_{m_n}$, $E_1,\ldots,E_n$, $F$ and $H$, if  $B_1\in \mathcal{L}(G_1,\ldots, G_{m_1};E_1), \ldots, B_n\in \mathcal{L}(G_{m_{n-1}+1},\ldots, G_{m_n};E_n)$, $t \in \mathcal{L}(F;H)$ and
$A \in \mathcal{H}(E_1,\ldots, E_n;F)$, then $t\circ A\circ(B_1,\ldots,B_n)$ belongs to $\mathcal{H}(G_1,\ldots, G_{m_n};H)$.

If there exist $p\in (0,1]$ and a map $\|\cdot\|_{\mathcal{H}} \colon \mathcal{H} \longrightarrow [0,\infty)$ such that:\\
(a) $\|\cdot\|_{\mathcal{H}}$ restricted to any component $\mathcal{H}(E_1,\ldots, E_n;F)$ is a $p$-norm;\\
(b) $\|I_n \colon \mathbb{K}^n\longrightarrow \mathbb{K}, I_n(\lambda_1,\ldots,\lambda_n)=\lambda_1\cdots\lambda_n\|_{\cal H} =1$ for every $n$;\\
(c) {\bf The hyper-ideal inequality:} If $A \in \mathcal{H}(E_1,\ldots, E_n;F)$, $B_1\in \mathcal{L}(G_{1},\ldots, G_{m_1};E_1)$, $\ldots$, $B_n\in \mathcal{L}(G_{m_{n-1}+1},\ldots, G_{m_n};E_n)$ and $t \in \mathcal{L}(F;H)$, then
\begin{equation}\|t\circ A\circ(B_1,\ldots,B_n)\|_{\mathcal{H}}\le\|t\|\cdot\|A\|_{\mathcal{H}}\cdot
\|B_1\|\cdots\|B_n\|,\label{eqhi}\end{equation}
then $(\mathcal{H},\|\cdot\|_{\mathcal{H}})$ is called a \textit{$p$-normed hyper-ideal}. Normed, Banach and $p$-Banach hyper ideals are defined in the obvious way.

If the hyper-ideal property and the hyper-ideal inequality are required to hold only for the composition with linear operators on the left hand side, that is, if they hold in the particular case where $m_1 = 1, m_2 = 2, \ldots, m_n = n$ and $B_1,B_2, \ldots, B_n$ are linear operators, then we recover the notion of multi-ideals (normed, $p$-normed, Banach, $p$-Banach multi-ideals). For the theory of multi-ideals we refer to \cite{botelho1, klausdomingo, joilson}.

Let us give an illustrative example of the failure of the factorization method in the generation of hyper-ideals:

\begin{ex}\label{ex1}\rm Starting the factorization method with the ideal $\pi_p$ of absolutely $p$-summing operators, we obtain the multi-ideal ${\cal L}_{d,p}$ of $p$-dominated multilinear operators \cite[Theorem 13]{pietsch} (for a proof, see \cite[Remark 3.3]{bpr1}). In \cite[Proposition 4.2(b)]{popass} it is proved that ${\cal L}_{d,p}$ fails to be a hyper-ideal.
\end{ex}

The following criterion shall be used twice:

\begin{teo}\label{criterio} {\rm \cite[Theorem 2.5]{ewerton}} A class $\mathcal{H}$ of continuous multilinear operators endowed with a map $\|\cdot\|_\mathcal{H}\colon\mathcal{H}\longrightarrow [0,+\infty)$ is a $p$-Banach hyper-ideal, $0<p\leq 1$, if and only if the following conditions are satisfied:\\
{\rm (i)} $I_n\in\mathcal{H}(^n\mathbb{K};\mathbb{K})$ and $\|I_n\|_{\mathcal{H}}=1$ for every $n \in \mathbb{N}$;\\
{\rm (ii)} If $(A_j)_{j=1}^\infty\subseteq\mathcal{H}(E_1,\ldots,E_n;F)$ is such that $\sum\limits_{j=1}^\infty\|A_j\|_\mathcal{H}^{p}<\infty$, then $$A:=\sum\limits_{j=1}^\infty A_j\in\mathcal{H}(E_1,\ldots,E_n;F)\ \mbox{and}\ \|A\|_\mathcal{H}^{p}\le\sum\limits_{j=1}^\infty \|A_j\|_\mathcal{H}^{p};$$
{\rm (iii)} $(\cal H, \|\cdot\|_{\cal H})$ enjoys the hyper-ideal property and the hyper-ideal inequality. \end{teo}

\begin{obs}\label{rema}\rm In \cite[Corollary 3.3]{ewerton} it is proved that every  hyper-ideal contains the finite rank multilinear operators.
\end{obs}

\section{The inequality method}\label{idtop}

The method that we introduce in this section generates new hyper-ideals, such as the class of strongly almost summing operators (cf. Corollary \ref{sas}), and recovers, as particular instances, well studied classes, such as strongly summing multilinear operators (cf. Corollary \ref{pssp}).

\begin{definition}\label{defpf}\rm (a) Let $0<p\le1$. By $\cal BAN$ we denote the class of all Banach spaces over $\mathbb{K}= \mathbb{R}$ or $\mathbb{C}$ and by $p{\rm-}{\cal BAN}$ the class of all $p$-Banach spaces over $\mathbb{K}$. A correspondence $$\mathcal{X}\colon \mathcal{BAN} \longrightarrow p{\rm -}\mathcal{BAN}$$ that associates to each Banach space $E$  a $p$-Banach space $(\mathcal{X}(E),\|\cdot\|_{\mathcal{X}(E)})$ is called a {\it $p$-sequence functor} if:\\
(i) $\mathcal{X}(E)$ is a linear subspace of $E^{\mathbb{N}}$  with the usual algebraic operations;\\
(ii) For all $x \in E$ and $j \in \mathbb{N}$, we have $(0,\ldots,0,x,0,\ldots)\in \mathcal{X}(E)$, where $x$ is placed at the $j$-th coordinate, and $\|(0,\ldots,0,x,0,\ldots)\|_{\mathcal{X}(E)}=\|x\|_E$.\\
(iii) For every $u\in \mathcal{L}(E;F)$ and every finite $E$-valued sequence $(x_j)_{j=1}^k := (x_1,\ldots,x_k,0,0,\ldots)$, $k \in \mathbb{N}$, it holds $$\|(u(x_j))_{j=1}^k\|_{\mathcal{X}(F)}\le \|u\|\cdot\|(x_j)_{j=1}^k\|_{\mathcal{X}(E)}.$$
When $p=1$ we simply say that $\cal X$ is a \textit{sequence functor}.\\
(b) Let $0 < p,q \leq 1$. A $p$-sequence functor $\mathcal{X}$ is \textit{scalarly dominated} by the $q$-sequence functor $\mathcal{Y}$ if, for every finite sequence $(\lambda_j)_{j=1}^k\subseteq\mathbb{K}$, $k \in \mathbb{N}$, we have $$\|(\lambda_j)_{j=1}^k\|_{\mathcal{X}(\mathbb{K})}\le \|(\lambda_j)_{j=1}^k\|_{\mathcal{Y}(\mathbb{K})}.$$
\end{definition}

The term {\it sequence functor} was used in \cite{laacariello} in a different sense.

\begin{ex}\label{fex}\rm (a) For $p >0$, the following correspondences are $p$-sequence functors (sequence functors if $p \geq 1$): \\
\indent $\bullet$ $E \longmapsto(\ell_p(E),\|\cdot\|_p)$ (absolutely $p$-summable sequences).\\
\indent$\bullet$ $E \longmapsto(\ell_p^w(E),\|\cdot\|_{w,p})$ (weakly $p$-summable sequences \cite{diestel}).\\
\indent$\bullet$ $E \longmapsto(\ell_p^u(E),\|\cdot\|_{w,p})$ (unconditionally $p$-summable sequences \cite{klauslivro}).\\
\indent$\bullet$ $E \longmapsto(\ell_p\langle E \rangle,\|\cdot\|_{\ell_p\langle E \rangle})$ (Cohen strongly $p$-summable sequences \cite{jamilson}).\\

(b) The following correspondences are sequence functors: \\
\indent$\bullet$ $E\longmapsto(c_0(E),\|\cdot\|_{\infty})$ (norm null sequences).\\
\indent$\bullet$ $E\longmapsto(c_0^w(E),\|\cdot\|_{\infty})$ (weakly null sequences).\\
\indent$\bullet$ $E\longmapsto(\ell_\infty(E),\|\cdot\|_{\infty})$ (bounded sequences).\\
\indent$\bullet$ $ E \longmapsto(Rad(E),\|\cdot\|_{Rad(E)})$, where $Rad(E)$ is the space of almost unconditionally summable $E$-valued sequences \cite[Chapter 12]{diestel} and
$$\|(x_j)_{j=1}^\infty\|_{Rad(E)}=\ap\int_0^1\an\sum\limits_{j=1}^\infty r_j(t)x_j\fn^2 dt\fp^{1/2},$$
where $(r_j)_{j=1}^\infty$ are the Rademacher functions.\\
(c) Among several other obvious dominations, for $0 < p < q$, $\ell_q(\cdot)$ is scalarly dominated by $\ell_p(\cdot)$, as well as  $\ell_q^w(\cdot)$ by  $\ell_p^w(\cdot)$, $\ell_q^u(\cdot)$ by $\ell_p^u(\cdot)$ and $\ell_q\langle\cdot\rangle$ by $\ell_p\langle\cdot\rangle$.\end{ex}

To introduce the inequality method we need the

\begin{lema} Let $\cal X$ be a $p$-sequence functor, $0 < p \leq 1$ and $(x_j^1)_{j=1}^k\subseteq E_1,\ldots,\ (x_j^n)_{j=1}^k\subseteq E_n$ be finite sequences. Then
$$\sup\limits_{T\in B_{\mathcal{L}(E_1,\ldots,E_n)}} \|(T(x_j^1,\ldots,x_j^n))_{j=1}^k\|_{\mathcal{X}(\mathbb{K})}<\infty.$$
\end{lema}

\begin{proof} For any $n$-linear form $T \in B_{\mathcal{L}(E_1,\ldots,E_n)}$,  \begin{align*}\|(T(x_j^1,\ldots&,x_j^n))_{j=1}^k\|_{\mathcal{X}(\mathbb{K})}^p=
\|(T(x_1^1,\ldots,x_1^n),0,\ldots,0,\ldots)+(0,T(x_2^1,\ldots,x_2^n),0,\ldots,0,\ldots)\\ &~~~+\cdots+(0,0,\ldots,0,T(x_k^1,\ldots,x_k^n),0,\ldots)\|_{\mathcal{X}(\mathbb{K})}^p \\&\le \|(T(x_1^1,\ldots,x_1^n),0,\ldots,0,\ldots)\|_{\mathcal{X}(\mathbb{K})}^p+ \|(0,T(x_2^1,\ldots,x_2^n),0,\ldots,0,\ldots)\|_{\mathcal{X}(\mathbb{K})}^p\\ &~~~+\cdots+\|(0,0,\ldots,0,T(x_k^1,\ldots,x_k^n),0,\ldots)\|_{\mathcal{X}(\mathbb{K})}^p
\\&=|T(x_1^1,\ldots,x_1^n)|^p+|T(x_2^1,\ldots,x_2^n)|^p+\cdots+|T(x_k^1,\ldots,x_k^n)|^p
\\&\le \|x_1^1\|^p\cdots\|x_1^n\|^p + \|x_2^1\|^p\cdots\|x_2^n\|^p + \cdots + \|x_k^1\|^p\cdots\|x_k^n\|^p.\end{align*}
The result follows because the latter term does not depend on $T$.
\end{proof}

\begin{definition}\label{dmdes}\rm Let $0 < p,q \leq 1$, $\mathcal{X}$ be a $p$-sequence functor and $\mathcal{Y}$ be a $q$-sequence functor. An $n$-linear operator $A\in\mathcal{L}(E_1,\ldots,E_n;F)$ is said to be {\it $(\mathcal{X};\mathcal{Y})$-summing} if there is a constant $C>0$ such that
\begin{equation} \label{eq:md}\|(A(x_j^1,\ldots,x_j^n))_{j=1}^k\|_{\mathcal{Y}(F)}\le C\cdot\sup\limits_{T\in B_{\mathcal{L}(E_1,\ldots,E_n)}} \|(T(x_j^1,\ldots,x_j^n))_{j=1}^k\|_{\mathcal{X}(\mathbb{K})},\end{equation} for every $k \in \mathbb{N}$ and all finite sequences $(x_j^1)_{j=1}^k\subseteq E_1,\ldots,\ (x_j^n)_{j=1}^k\subseteq E_n$. In this case we write $A \in(\mathcal{X};\mathcal{Y})(E_1,\ldots,E_n;F)$ and define $$\|A\|_{(\mathcal{X};\mathcal{Y})}=\inf\{C >0 : C~\mbox{satisfies} \ (\ref{eq:md})\}.$$\end{definition}

To prove that the inequality method generates hyper-ideals we need the following result. The proof is standard and we omit it.

\begin{lema}\label{nlem}Let $\mathcal{X}$ be a $p$-sequence functor, $0 < p \le 1$, and $n\in\mathbb{N}$. If $\sum\limits_{j=1}^\infty x_j^1, \ldots, \sum\limits_{j=1}^\infty x_j^n$ are convergent series in $E$, then the series $\sum\limits_{j=1}^\infty (x_j^1,\ldots,x_j^n,0,0,\ldots)$ converges in $\mathcal{X}(E)$ and
$$\sum\limits_{j=1}^\infty (x_j^1,\ldots,x_j^n,0,0,\ldots)=\ap\sum\limits_{j=1}^\infty x_j^1, \ldots, \sum\limits_{j=1}^\infty x_j^n,0,0,\ldots\fp {\rm ~in~} {\cal X}(E).$$\end{lema}

\begin{teo}\label{pmdes}Let $0 < p,q \leq 1$ and $\mathcal{Y}$ be a $q$-sequence functor scalarly dominated by the $p$-sequence functor $\mathcal{X}$. Then $((\mathcal{X};\mathcal{Y}),\|\cdot\|_{(\mathcal{X};\mathcal{Y})})$ is a $q$-Banach hyper-ideal.\end{teo}

\begin{proof} We prove that the conditions of Theorem \ref{criterio} are fulfilled. \\
(i) It is routine to prove that, for each $n\in \mathbb{N}$,  $I_n\in(\mathcal{X};\mathcal{Y})(^n\mathbb{K};\mathbb{K})$. As $I_n \in B_{{\cal L}(^n\mathbb{K};\mathbb{K})}$ and $\cal Y$ is scalarly dominated by $\cal X$, we get  $\|I_n\|_{(\mathcal{X};\mathcal{Y})}\leq 1$. Let $C$ be a constant working in (\ref{eq:md}) for $I_n$. Choosing $k =  x_1^1, \ldots, x_n^1 =1$ we obtain $C \geq 1$, from which follows $\|I_n\|_{(\mathcal{X};\mathcal{Y})}= 1$.  \\
(ii) Let $(A_i)_{i=1}^\infty\subseteq (\mathcal{X};\mathcal{Y})(E_1,\ldots,E_n;F)$ be such that $\sum\limits_{i=1}^\infty \|A_i\|_{(\mathcal{X};\mathcal{Y})}^q<\infty$. We are supposed to show that
\begin{equation}\label{eqseries} A:=\sum\limits_{i=1}^\infty A_i \in (\mathcal{X};\mathcal{Y})(E_1,\ldots,E_n;F) {\rm ~~and~~} \|A\|_{(\mathcal{X};\mathcal{Y})}^q\le\sum\limits_{i=1}^\infty \|A_i\|_{(\mathcal{X};\mathcal{Y})}^q.\end{equation}
First note that, for each operator $B \in (\mathcal{X};\mathcal{Y})(E_1,\ldots,E_n;F)$ and all $x_j \in E_j$, $j = 1, \ldots, n$, we have  \begin{eqnarray*}\|B(x_1,\ldots,x_n)\|_F&=&\|(B(x_1,\ldots,x_n),0,\ldots,0,\ldots)\|_{\mathcal{Y}(F)}\\&\le& C\cdot\sup\limits_{T\in B_{\mathcal{L}(E_1,\ldots,E_n)}}\|(T(x_1,\ldots,x_n),0,\ldots,0,\ldots)\|_{\mathcal{X}(\mathbb{K})}\\&=&
C\cdot\sup\limits_{T\in B_{\mathcal{L}(E_1,\ldots,E_n)}}|T(x_1,\ldots,x_n)|\le C\cdot\|x_1\|\cdots\|x_n\|,\end{eqnarray*}
for each constant $C$ working in (\ref{eq:md}) for $B$. It follows that $\|B\|\le\|B\|_{(\mathcal{X};\mathcal{Y})}$. Hence, as $q \leq 1$, the series $\sum\limits_{i=1}^\infty A_i$ is absolutely convergent in the Banach space ${\cal L}(E_1, \ldots, E_n;F)$, therefore convergent, say $A:=\sum\limits_{i=1}^\infty A_i \in {\cal L}(E_1, \ldots, E_n;F)$. Given $k \in \mathbb{N}$ and finite sequences $(x_j^1)_{j=1}^k\subseteq E_1,\ldots,\ (x_j^n)_{j=1}^k\subseteq E_n$, Lemma \ref{nlem} gives \begin{align*}\|(A(x_j^1,\ldots,x_j^n))&_{j=1}^k\|_{\mathcal{Y}(F)}^q= \an\ap\sum\limits_{i=1}^\infty A_i(x_j^1,\ldots,x_j^n)\fp_{j=1}^k\fn_{\mathcal{Y}(F)}^q\\&=\an\sum\limits_{i=1}^\infty (A_i(x_j^1,\ldots,x_j^n))_{j=1}^k\fn_{\mathcal{Y}(F)}^q\le \sum\limits_{i=1}^\infty \|(A_i(x_j^1,\ldots,x_j^n))_{j=1}^k\|_{\mathcal{Y}(F)}^q \\&= \ap\sum\limits_{i=1}^\infty \|A_i\|_{(\mathcal{X};\mathcal{Y})}^q\fp\cdot\ap\sup\limits_{T\in B_{\mathcal{L}(E_1,\ldots,E_n)}} \|(T(x_j^1,\ldots,x_j^n))_{j=1}^k\|_{\mathcal{X}(\mathbb{K})}\fp^q,\end{align*}
proving (\ref{eqseries}).\\
(iii) Let $\le m_1<\cdots<m_n$ be naturals numbers, $A \in (\mathcal{X};\mathcal{Y})(E_1,\ldots,E_n;F)$, $B_1\in \mathcal{L}(G_1,\ldots, G_{m_1};E_1)$, $\ldots, B_n\in \mathcal{L}(G_{m_{n-1}+1},\ldots, G_{m_n};E_n)$ and $t\in\mathcal{L}(F;H)$. Of course we can assume that $B_l\neq0$ for $l=1,\ldots,n$. For every $T\in B_{\mathcal{L}(E_1,\ldots,E_n)}$,  $$T\circ\ap\frac{B_1}{\|B_1\|},\ldots,\frac{B_n}{\|B_n\|}\fp \in B_{\mathcal{L}(G_1,\ldots,G_{m_n})},$$
thus
\begin{align*}&
\|(t\circ A \circ(B_1,\ldots,B_n)(x_j^1,\ldots,x_j^n))_{j=1}^k\|_{\mathcal{Y}(H)} \\&\le\|t\|\cdot\|(A(B_1(x_j^1,\ldots,x_j^{m_1}),\ldots,B_n(x_j^{m_{n-1}+1},\ldots,x_j^{m_n})))_{j=1}^k\|_{\mathcal{Y}(F)}\\&\le
\|t\|\cdot\|A\|_{(\mathcal{X};\mathcal{Y})}\cdot \sup\limits_{T\in B_{\mathcal{L}(E_1,\ldots,E_n)}}\|(T(B_1(x_j^1,\ldots,x_j^{m_1}),\ldots, B_n(x_j^{m_{n-1}+1},\ldots,x_j^{m_n})))_{j=1}^k\|_{\mathcal{X}(\mathbb{K})}\\&=
\|t\|\cdot\|A\|_{(\mathcal{X};\mathcal{Y})}\cdot\|B_1\|\cdots\|B_n\|\cdot\\& {  } ~~~~~~~\sup\limits_{T\in B_{\mathcal{L}(E_1,\ldots,E_n)}}\an\ap T\ap \frac{B_1}{\|B_1\|}(x_j^1,\ldots,x_j^{m_1}),\ldots, \frac{B_n}{\|B_n\|}(x_j^{m_{n-1}+1},\ldots,x_j^{m_n})\fp\fp_{j=1}^k\fn_{\mathcal{X}(\mathbb{K})}
\\&\le\|t\|\cdot\|A\|_{(\mathcal{X};\mathcal{Y})}\cdot\|B_1\|\cdots\|B_n\| \cdot \sup\limits_{S\in B_{\mathcal{L}(G_1,\ldots,G_{m_n})}} \|(S(x_j^1,\ldots,x_j^{m_n}))_{j=1}^k\|_{\mathcal{X}(\mathbb{K})},\end{align*}
where the first inequality follows from condition \ref{defpf}(iii).
It follows that $t\circ A \circ(B_1,\ldots,B_n)\in(\mathcal{X};\mathcal{Y})(G_1,\ldots,G_{m_n};H)$ and $$\|t\circ A \circ(B_1,\ldots,B_n)\|_{(\mathcal{X};\mathcal{Y})}\le \|t\|\cdot\|A\|_{(\mathcal{X};\mathcal{Y})}\cdot\|B_1\|\cdots\|B_n\|.$$\end{proof}

Next we provide two examples of hyper-ideals generated by the inequality method.

In the first example we show that the inequality method recovers an important well studied class as a particular instance. The class of dominated multilinear operators was introduced by Pietsch \cite{pietsch} as a first attempt to generalize the classical ideal of absolutely summing linear operators to the multilinear setting. Although several other classes of {\it absolutely summing multilinear operators} have appeared, the class of dominated multilinear operators keeps being studied to this day. Among other recent developments, Popa \cite{popass} proved that the class of dominated multilinear operators fails to be a hyper-ideal. It is a natural question to ask if there is room in the realm of Banach hyper-ideals for a multilinear generalization of the Banach ideal of absolutely $p$-summing linear operators. We found the answer in the following class introduced by Dimant \cite{dimant}:

\begin{definition}\rm For $0< p<\infty$, an $n$-linear operator $A\in\mathcal{L}(E_1,\ldots,E_n;F)$ is {\it strongly $p$-summing} if there is $C>0$ such that \begin{equation} \label{eq:ssp}\ap\sum\limits_{i=1}^k\|A(x_i^{1},\ldots,x_i^{n})\|^p\fp^{1/p}\le C\cdot\sup\limits_{T\in B_{\mathcal{L}(E_1,\ldots,E_n)}}\ap\sum\limits_{i=1}^k|T(x_i^{1},\ldots,x_i^{n})|^{p}\fp^{1/p}\end{equation}
for all $(x_i^{l})_{i=1}^k\subseteq E_l$, $l=1,\ldots,n$, $k \in \mathbb{N}$. In this case we write $A\in\mathcal{L}_{ss}^p(E_1,\ldots,E_n;F)$ and define
$$\|A\|_{ss,p}=\inf\{C >0: C \ \mbox{satisfies}\ (\ref{eq:ssp})\}.$$\end{definition}

It is clear that the linear component of $\mathcal{L}_{ss}^p$ recovers the ideal of absolutely $p$-summing linear operators. Making ${\cal Y}(E) = {\cal X}(E) = \ell_p(E)$ for every Banach space $E$ in  Theorem \ref{pmdes}, we obtain:

\begin{cor}\label{pssp}The class $(\mathcal{L}_{ss}^p,\|\cdot\|_{ss,p})$ of strongly $p$-summing multilinear operators is a $p$-Banach hyper-ideal for $0<p<1$ and a Banach hyper-ideal for $p \geq 1$.\end{cor}

The second example of a hyper-ideal generated by the inequality method is a new class. The ideal of almost summing linear operators was introduced in \cite{diestel} and several classes of {\it almost summing multilinear operators} have been studied, see, e.g. \cite{bnachr, bbj, danieljoilson, danielmarcela, popaarchiv}. Such classes often fail to be hyper-ideals:

\begin{ex}\rm Let ${\cal L}_{al.s}$ denote the class of almost summing multilinear operators introduced in \cite{bnachr, bbj}. By \cite[Example 4.3(2)]{bnachr}, there are multilinear operators of finite rank not belonging to ${\cal L}_{al.s}$, so ${\cal L}_{al.s}$ fails to be a hyper-ideal by Remark \ref{rema}.
 \end{ex}

Our second example is a hyper-ideal generated by the inequality method that generalizes the ideal of almost summing linear operators to the multilinear setting. Remember that $(Rad(\cdot), \|\cdot\|_{Rad(E)})$ denotes the sequence functor of almost unconditionally summable sequences (cf. Example \ref{fex}(b)).

\begin{definition}\rm Let $p\ge1$. We say that an $n$-linear operator $A\in \mathcal{L}(E_1,\ldots,E_n;F)$ is {\it strongly almost $p$-summing} if there is a constant $C>0$ such that, for all $k\in\mathbb{N}$ and $(x_j^1)_{j=1}^k\subseteq E_1,\ldots, (x_j^n)_{j=1}^k\subseteq E_n$, \begin{equation}\label{eq:fqps}\left\| (A(x_j^1,\ldots,x_j^n))_{j=1}^k\right\|_{Rad(F)}\le C\cdot\sup\limits_{T\in B_{\mathcal{L}(E_1,\ldots,E_n)}} \ap\sum\limits_{j=1}^k|(T(x_j^1,\ldots,x_j^n))_{j=1}^k|^p\fp^{1/p}.\end{equation} In this case we write $A\in\mathcal{L}_{sas,p}(E_1,\ldots,E_n;F)$ and define  $$\|A\|_{\mathcal{L}_{sas,p}}=\inf\{C >0: C \ \mbox{satisfies} \ \eqref{eq:fqps}\}.$$\end{definition}

\begin{cor}\label{sas}If $0<p\le2$, then $(\mathcal{L}_{sas,p},\|\cdot\|_{\mathcal{L}_{sas,p}})$ is an Banach hyper-ideal.\end{cor}
\begin{proof}Note that $\mathcal{L}_{sas,p}$ is precisely the class of $(\ell_p(\cdot);Rad(\cdot))$-summing multilinear operators. As $\ell_p(\cdot)$ is a $p$-sequence functor, $Rad(\cdot)$ is a sequence functor and $Rad(\cdot)$ is scalarly dominated by $\ell_p(\cdot)$ (remember that $Rad(\mathbb{K}) = \ell_2$ isometrically and $0<p\le2$), the result follows from Theorem \ref{pmdes}.\end{proof}

\begin{obs}\rm A method to generate multi-ideals (not hyper-ideals), related to the method introduced in this section, is presented in \cite{diana}. However, an important assumption is missing there. More precisely, for Theorem \cite[Theorem 3]{diana} to be true, the linear stability of the underlying sequence spaces is required, that is, condition \ref{defpf}(iii) must be added to the assumptions of \cite[Theorem 3]{diana}.\end{obs}

\section{The boundedness method}\label{action}
As mentioned in the Introduction, some methods of generating multi-ideals starting with a given operator ideal are not effective to generate hyper-ideals (cf. Example \ref{ex1}). In this section, inspired by the polynomial case studied in \cite{aronrueda2}, we introduce a method to generate hyper-ideals starting with a given operator ideal.\\
\indent The notion of $\cal I$-bounded sets, where $\cal I$ is an operator ideal, was introduced by Stephani \cite{stephani}; recent developments can be found, e.g.,  in  \cite{aronrueda2, aronrueda1, gg1, lassalleturco}.

\begin{definition}\rm Let $\mathcal{I}$ be an operator ideal. A subset $K$ of a Banach space $F$ is said to be {\it $\mathcal{I}$-bounded} is there are a Banach space $H$ and an operator $u\in\mathcal{I}(H;F)$ such that $K\subseteq u(B_H).$ The collection of all $\cal I$-bounded subsets of $F$ is denoted by $C_{\mathcal{I}}(F)$.\end{definition}

Aron and Rueda \cite{aronrueda2} used the concept of $\cal I$-bounded set to define an ideal of homogeneous polynomials. In this section we show that, proceeding for multilinear operators as Aron and Rueda proceeded for polynomials, we end up with a Banach hyper-ideal.

\begin{definition}\rm Let $(\mathcal{I}, \|\cdot\|_{\cal I})$ be a $p$-normed operator ideal, $0 < p \leq 1.$ We say that a multilinear operator $A\in\mathcal{L}(E_1,\ldots,E_n;F)$ is {\it $\mathcal{I}$-bounded} if $A(B_{E_1}\times\cdots\times B_{E_n})\in C_{\mathcal{I}}(F)$, that is, if
there are a Banach space $H$ and an operator $u\in\mathcal{I}(H;F)$ such that \begin{equation}\label{eq:milim}A(B_{E_1}\times\cdots\times B_{E_n}) \subseteq u(B_H).\end{equation}
In this case we write $A\in{\cal L}_{\langle\mathcal{I}\rangle}(E_1,\ldots,E_n;F)$ and define
$$\|A\|_{\mathcal{L}_{\langle\mathcal{I}\rangle}}=\inf\{\|u\|_{\mathcal{I}}: \ u \ \mbox{satisfies} \ \eqref{eq:milim}\}.$$
\end{definition}

\begin{teo}\label{proplimhi}Let $0 < p \leq 1$ and $(\mathcal{I},\|\cdot\|_\mathcal{I})$ be a $p$-normed ($p$-Banach) operator ideal. Then $(\mathcal{L}_{\langle\mathcal{I}\rangle},\|\cdot\|_{\mathcal{L}_{\langle\mathcal{I}\rangle}})$ is a $p$-normed ($p$-Banach) hyper-ideal.\end{teo}

\begin{proof}We omit the proof of the incomplete case. Let us apply Theorem \ref{criterio} to prove that $(\mathcal{L}_{\langle\mathcal{I}\rangle},\|\cdot\|_{\mathcal{L}_{\langle\mathcal{I}\rangle}})$ is a $p$-Banach hyper-ideal whenever $(\mathcal{I},\|\cdot\|_\mathcal{I})$ is a $p$-Banach operator ideal.\\
(i) As $Id_{\mathbb{K}} \in {\cal I}(\mathbb{K};\mathbb{K})$ and $\|Id_{\mathbb{K}}\|_{\cal I} = 1$, it follows easily that  $I_n\in\mathcal{L}_{\langle\mathcal{I}\rangle}(^n\mathbb{K};\mathbb{K})$ and $\|I_n\|_{\mathcal{L}_{\langle\mathcal{I}\rangle}}\leq 1$. For all $H$ and $u \in {\cal I}(H;\mathbb{K})$ such that $I_n\left((B_{\mathbb{K}})^n\right) \subseteq u(B_H)$, choosing $z \in B_H$ such that $u(z) = 1$ we get
$$\|u\|_{\cal I} \geq \|u\| \geq |u(z)| = 1,$$
from which we conclude that $\|I_n\|_{\mathcal{L}_{\langle\mathcal{I}\rangle}}= 1$.\\
(ii) Let $(A_j)_{j=1}^\infty\subseteq\mathcal{L}_{\langle\mathcal{I}\rangle}(E_1,\ldots,E_n;F)$ be such that $\sum\limits_{j=1}^\infty\|A_j\|_{\mathcal{L}_{\langle\mathcal{I}\rangle}}^p<\infty$. Let $\varepsilon>0$. For each $j\in \mathbb{N}$, the set $A_j(B_{E_1}\times\cdots\times B_{E_n})$ is $\mathcal{I}$-bounded, thus there exist a Banach space $H_j$ and an operator $u_j\in\mathcal{I}(H_j;F)$ such that $$A_j(B_{E_1}\times\cdots\times B_{E_n})\subseteq u_j(B_{H_j})$$ and $\|u_j\|_{\mathcal{I}}<(1+\varepsilon)\|A_j\|_{\mathcal{L}_{\langle\mathcal{I}\rangle}}$. So, $$\|A_j\|\le\|u_j\|\le\|u_j\|_{\mathcal{I}}<(1+\varepsilon)\|A_j\|_{\mathcal{L}_{\langle\mathcal{I}\rangle}}.$$ Making $\varepsilon \longrightarrow 0$ we obtain  $\|A_j\|\le\|A_j\|_{\mathcal{L}_{\langle\mathcal{I}\rangle}}$. As $p \leq 1$, we conclude that the series $\sum\limits_{j=1}^\infty A_j$ is absolutely convergent in the Banach space ${\cal L}(E_1, \ldots, E_n;F)$, hence convergent, say
$A:=\sum\limits_{j=1}^\infty A_j\in\mathcal{L}(E_1,\ldots,E_n;F)$. Let $$H:=\ap\bigoplus\limits_{j=1}^{\infty} H_j,\|\cdot\|_\infty\fp$$
be the Banach space of bounded sequences $(x_j)_{j=1}^\infty$, where $x_j\in H_j$ for every $j$, endowed with the sup norm.
Letting $\pi_j\colon H\longrightarrow H_j$, $j \in \mathbb{N}$, be the canonical projections and defining $$v_j\colon H\longrightarrow F~,~v_j=u_j\circ\pi_j,$$
by the ideal property of $\cal I$ we have that each $v_j\in\mathcal{I}(H;F)$ and $$\sum\limits_{j=1}^\infty\|v_j\|_{\mathcal{I}}^p=\sum\limits_{j=1}^\infty\|u_j\circ\pi_j\|_{\mathcal{I}}^p\le \sum\limits_{j=1}^\infty\|u_j\|_{\mathcal{I}}^p<(1+\varepsilon)^p\cdot\sum\limits_{j=1}^\infty \|A_j\|_{\mathcal{L}_{\langle\mathcal{I}\rangle}}^p<\infty.$$ Since  $(\mathcal{I},\|\cdot\|_{\cal I})$ is a $p$-Banach ideal, it follows that
$$u:=\sum\limits_{j=1}^\infty v_j\in\mathcal{I}(H;F)~\ \mbox{and}~\ \|u\|_{\mathcal{I}}^p< (1+\varepsilon)^p\cdot\sum\limits_{j=1}^\infty\|A_j\|_{\mathcal{L}_{\langle\mathcal{I}\rangle}}^p.$$
Given $y\in A(B_{E_1}\times\cdots\times B_{E_n})$, choose  $x_l\in B_{E_l}$, $l=1,\ldots,n$, such that $$y=A(x_1,\ldots,x_n)=\sum\limits_{j=1}^\infty A_j(x_1,\ldots,x_n).$$ As $A_j(B_{E_1}\times\cdots\times B_{E_n})\subseteq u_j(B_{H_j}),$ for every $j$ there is $z_j\in B_{H_j}$ such that $A_j(x_1,\ldots,x_n)=u_j(z_j).$ Hence, $$y=\sum\limits_{j=1}^\infty A_j(x_1,\ldots,x_n)=\sum\limits_{j=1}^\infty u_j(z_j)=\sum\limits_{j=1}^\infty u_j\circ\pi_j((z_k)_{k=1}^\infty)= \sum\limits_{j=1}^\infty v_j((z_k)_{k=1}^\infty)= u((z_k)_{k=1}^\infty),$$ where $(z_k)_{k=1}^\infty\in B_H$ because $\|z_j\| \leq 1$ for all $j$. Then $y\in u(B_H)$, which means that $A(B_{E_1}\times\cdots\times B_{E_n})\in C_{\mathcal{I}}(F)$. This proves that  $A\in\mathcal{L}_{\langle\mathcal{I}\rangle}(E_1,\ldots,E_n;F)$ and $$\|A\|_{\mathcal{L}_{\langle\mathcal{I}\rangle}}^p\le \|u\|_{\mathcal{I}}^p< (1+\varepsilon)^p\cdot\sum\limits_{j=1}^\infty\|A_j\|_{\mathcal{L}_{\langle\mathcal{I}\rangle}}^p.$$ Just make $\varepsilon \longrightarrow 0$ to obtain $\|A\|_{\mathcal{L}_{\langle\mathcal{I}\rangle}}^p\le \sum\limits_{j=1}^\infty\|A_j\|_{\mathcal{L}_{\langle\mathcal{I}\rangle}}^p$.\\
(iii) Let $t\in\mathcal{L}(F;G)$, $A\in\mathcal{L}_{\langle\mathcal{I}\rangle}(E_1,\ldots,E_n;F)$ and $B_1\in\mathcal{L}(G_{1},\ldots,G_{m_1};E_1)$,$\ldots$, $B_n\in\mathcal{L}(G_{m_{n-1}+1},\ldots,G_{m_n};E_n)$, where $1\le m_1<\cdots<m_n$. By the definition of ${\cal L}_{\langle{\cal I}\rangle}$ there are a Banach space $H$ and an operator $u\in\mathcal{I}(H;F)$ such that \begin{equation}\label{limi}A(B_{E_1}\times\cdots\times B_{E_n})\subseteq u(B_H).\end{equation} Of course we can assume $B_l\neq0$ for $l=1,\ldots,n$, and in this case  $$\frac{B_l}{\|B_l\|}\left(B_{G_{m_{l-1}+1}}\times\cdots\times B_{G_{m_l}}\right)\subseteq B_{E_l},$$ what gives $$A\circ\ap\frac{B_1}{\|B_1\|},\ldots,\frac{B_n}{\|B_n\|}\fp(B_{G_1}\times\cdots\times B_{G_{m_n}})\subseteq u(B_H).$$ Hence, $$t\circ A\circ(B_1,\ldots,B_n)(B_{G_1}\times\cdots\times B_{G_{m_n}})\subseteq (\|B_1\|\cdots\|B_n\|t\circ u)(B_H).$$ This proves that
$$t\circ A\circ(B_1,\ldots,B_n)\in\mathcal{L}_{\langle\mathcal{I}\rangle}(G_1,\ldots,G_{m_n};G),$$ because $\|B_1\|\cdots\|B_n\|t\circ u\in\mathcal{I}(H;G)$, and that $$\|t\circ A\circ(B_1,\ldots,B_n)\|_{\mathcal{L}_{\langle\mathcal{I}\rangle}} \le\|t\|\cdot\|u\|_{\mathcal{I}}\cdot\|B_1\|\cdots\|B_n\|.$$ Taking the infimum over all operators $u$ satisfying (\ref{limi}) we get the desired hyper-ideal inequality.\end{proof}

\begin{ex}\rm Let $\cal K$ and $\cal W$ denote the ideals of compact and weakly compact linear operators. Reasoning as in \cite[Example 3.1]{aronrueda2}, we see that $C_{\cal K}(E)$ is the collection of relatively compact subsets of $E$ and that $C_{\cal W}(E)$ is the collection of relatively weakly compact subsets of $E$. Then the classes ${\cal L}_{\cal K}$ of compact multilinear operators and ${\cal L}_{\cal W}$ of weakly compact multilinear operators are closed hyper-ideals (that is, Banach hyper-ideals with the usual sup norm) by Theorem \ref{proplimhi}
\end{ex}

The information in the example above was obtained in \cite{ewerton} by a different reasoning. To give new applications of the method introduced in this section, we consider the following concept introduced by Sinha and Karn \cite{sinha}, which has been playing an important role in the theory of operator ideals (cf. \cite{lassalleturco, pietschpams} and references therein) and in the study of variants of the approximation property (cf. \cite{choi, lassalleturco} and references therein).

\begin{definition}\rm Given $1 \leq p < \infty$, let $p'$ be given by $\frac{1}{p}+\frac{1}{p'}=1$. A subset $K$ of a Banach space $E$ is said to be {\it relatively $p$-compact} if there is a sequence $(x_j)_{j=1}^\infty \in \ell_p(E)$ such that
\begin{equation}K\subseteq\left\{\sum\limits_{j=1}^\infty \lambda_jx_j :  (\lambda_j)_{j=1}^\infty\in B_{\ell_{p'}} \right\}.\label{eq:apcomp}\end{equation}
\end{definition}

The definition below is the multilinear counterpart of the polynomial case studied by Aron and Rueda \cite{aronrueda3}.

\begin{definition}\rm An $n$-linear operator $A\in\mathcal{L}(E_1,\ldots,E_n;F)$ is said to be {\it $p$-compact}, $p \geq 1$, if $A(B_{E_1}\times\cdots\times B_{E_n})$ is a relatively $p$-compact subset of $F$. In this case we write  $A\in\mathcal{L}_{\mathcal{K}_p}(E_1,\ldots,E_n;F)$ and define
$$\|A\|_{\mathcal{L}_{\mathcal{K}_p}}=\inf\left\{\|(x_j)_{j=1}^\infty\|_p :  (x_j)_{j=1}^\infty \mbox{~satisfies} \ (\ref{eq:apcomp}){\rm ~for~} A(B_{E_1}\times\cdots\times B_{E_n}) \right\}.$$\end{definition}

The linear component of $\mathcal{L}_{\mathcal{K}_p}$ recovers the intensively studied ideal ${\cal K}_p$ of $p$-compact linear operators.

\begin{prop}\label{pcomilimex} For every $p \geq 1$, $\mathcal{L}_{\mathcal{K}_p} = \mathcal{L}_{\langle\mathcal{K}_p\rangle}$ and $\|\cdot\|_{\mathcal{L}_{\mathcal{K}_p}}=\|\cdot\|_{\mathcal{L}_{\langle\mathcal{K}_p\rangle}}$. In particular, the class $\mathcal{L}_{\mathcal{K}_p}$ of $p$-compact multilinear operators is a Banach hyper-ideal.\end{prop}

\begin{proof} In \cite[Example 3.1]{aronrueda2} it is proved that, for every Banach space $E$, $C_{{\cal K}_p}(E)$ coincides with the collection of relatively $p$-compact subsets of $E$. So, $\mathcal{L}_{\mathcal{K}_p} = \mathcal{L}_{\langle\mathcal{K}_p\rangle}$. The inequality $\|\cdot\|_{\mathcal{L}_{\mathcal{K}_p}}\leq \|\cdot\|_{\mathcal{L}_{\langle\mathcal{K}_p\rangle}}$ follows easily from the definitions. Given $A\in\mathcal{L}_{\mathcal{K}_p}(E_1,\ldots,E_n;F)$ and a sequence $(x_j)_{j=1}^\infty\in\ell_p(F)$ satisfying (\ref{eq:apcomp}) for $A(B_{E_1}\times\cdots\times B_{E_n})$, the linear operator
$$u\colon \ell_{p'}\longrightarrow F~,~u((\lambda_j)_{j=1}^\infty)=\sum\limits_{j=1}^\infty \lambda_jx_j,$$
is well defined by H\"older's inequality. In
$$A(B_{E_1}\times\cdots\times B_{E_n})\subseteq  \left\{\sum\limits_{j=1}^\infty \lambda_jx_j :  (\lambda_j)_{j=1}^\infty\in B_{\ell_{p'}} \right\}= u(B_{\ell_{p'}}),$$
the equality shows that $u$ is a $p$-compact linear operator and the inclusion gives
$$\|A\|_{\mathcal{L}_{\langle\mathcal{K}_p\rangle}}\le\|u\|_{\mathcal{K}_p}\le\|(x_j)_{j=1}^\infty\|_p.$$ Taking the infimum over all such sequences $(x_j)_{j=1}^\infty$ we conclude that $\|A\|_{\mathcal{L}_{\langle\mathcal{K}_p\rangle}}\le\|A\|_{\mathcal{L}_{\mathcal{K}_p}}$.

Now the second assertion follows from Theorem \ref{proplimhi}.
\end{proof}

\bigskip

\noindent{\bf Acknowledgement.} We thank Professors Mary Lilian Louren\c co for making our collaboration possible and Vin\'icius V. F\'avaro for his helpful suggestions.

\vspace{2em}

\noindent Faculdade de Matem\'atica\\
Universidade Federal de Uberl\^andia\\
38.400-902 -- Uberl\^andia, Brazil\\
e-mails: botelho@ufu.br, ewerton@powerline.com.br.

\end{document}